\documentclass[10pt]{article}

\usepackage[nottoc,notlot,notlof]{tocbibind}
\usepackage{amsmath,amssymb,amsfonts,amsthm}
\usepackage{latexsym}
\usepackage{tikz}
\usepackage{algorithm}
\usepackage{algpseudocode}
\usepackage[normalem]{ulem}
\usetikzlibrary{automata,positioning}
\usetikzlibrary{chains,fit,shapes}
\usetikzlibrary{calc}
\usetikzlibrary{arrows}
\usepackage{float}
\usetikzlibrary{positioning,calc}
\usetikzlibrary{graphs}
\usetikzlibrary{graphs.standard}
\usetikzlibrary{arrows,decorations.markings}
\usepackage{multicol}
\usepackage{xcolor}
\if@thmsec
  	\newtheorem{theorem}{Theorem}[section]
\else
   \newtheorem{theorem}{Theorem}
\fi

\newtheorem{lemma}[theorem]{Lemma}
\newtheorem{prop}[theorem]{Proposition}
\newtheorem{obs}[theorem]{Observation}
\newtheorem{cor}[theorem]{Corollary}

  \newtheorem{dnt}[theorem]{Definition}
\newtheorem{exm}[theorem]{Example}

\numberwithin{equation}{section}
\title{Word-Representable Co-Bipartite Graphs: Vertex Ordering, Representation Number, Speed, and Entropy}

\author{\hspace{1cm} Biswajit Das \ and Ramesh Hariharasubramanian \\ 
{{\footnotesize d.biswajit@iitg.ac.in},\ {\footnotesize  ramesh\_h@iitg.ac.in}}\\{\footnotesize Department of Mathematics, Indian Institute of Technology Guwahati, Guwahati, Assam 781039, India}}

\begin{document}
	\maketitle
	
	\begin{abstract}
A graph $G(V, E)$ is \textit{word-representable} if there exists a word $w$ over the alphabet $V$ such that for distinct letters $x,y\in V$, $x$ and $y$ alternate in $w$ if and only if they are adjacent in $G$. In general, determining whether a graph is word-representable is an NP-complete problem.  A graph is \textit{co-bipartite} if its complement is bipartite. Therefore, the vertex set of a co-bipartite graph can be partitioned into two disjoint subsets $X$ and $Y$ such that the subgraphs induced by $X$ and $Y$ are cliques.  
 
In this paper, we obtain necessary and sufficient conditions for a co-bipartite graph to be word-representable in terms of a vertex ordering. Based on this ordering, we study the representation number of word-representable co-bipartite graphs and analyse the speed and entropy of this graph class. We show that the representation number of any word-representable co-bipartite graph is at most $3$, and that permutation graphs are the only co-bipartite graphs with representation number $2$. We prove that the speed is $2^{O(n \log n)}$ and the entropy is $0$. This provides an asymptotic bound on the number of labelled graphs in this class, which is significantly smaller than the known bound for the class of all co-bipartite graphs. These results provide a better understanding of the structure and enumeration of word-representable co-bipartite graphs and show that vertex ordering is an effective tool for studying this class. 
	\end{abstract}

\section{Introduction}
The concept of word-representable graphs was introduced by Sergey Kitaev in connection with the study of the Perkins semigroup~\cite{kitaev2008word}. Since then, this notion has been studied widely and has found connections to algebra, graph theory, computer science, and combinatorics on words. Word-representable graphs generalise several well-known graph classes, including circle graphs, comparability graphs, and all $3$-colourable graphs. These graphs also arise in scheduling problems with repeating tasks, where certain pairs of tasks are required to alternate over time, and these alternation requirements naturally form an undirected graph.  A general overview of the background and motivation for studying word-representable graphs is given in~\cite{kitaev2015words}.

Word-representability has been studied for several classes of graphs. In \cite{kitaev2008representable}, outerplanar graphs and prism graphs were shown to be word-representable, while odd wheel graphs were proved to be non-word-representable. The word-representability of crown graphs and $k$-cube graphs was investigated in~\cite{glen2018representation} and~\cite{broere24k}, respectively. The word-representability of split graphs was studied in~\cite{iamthong2022word}, \cite{kitaev2021word}, \cite{iamthong2023semi}, and~\cite{chen2022representing}. In particular, \cite{kitaev2024semi} studied the word-representability of split graphs using vertex labelling and also provided a polynomial-time recognition algorithm for word-representable split graphs. More recently, the authors of~\cite{chen2025word} studied the $K_m-K_n$ graph (co-bipartite graphs) and identified forbidden induced subgraphs for $K_m-K_3$ and $K_m-K_4$.

The class of co-bipartite graphs contains both word-representable and non-word-representable graphs. In particular, the graphs $G_1$ and $G_2$ are minimal, with respect to the number of vertices, non-word-representable co-bipartite graphs, as shown in Figure~\ref{nonRepTri1}.

\begin{figure}[H]
\begin{center}
\begin{tabular}{ccccc}
%\begin{tikzpicture}[node distance=1cm,auto,main node/.style={fill,circle,draw,inner sep=0pt,minimum size=5pt}]
\begin{tikzpicture}[
scale=.8,
transform shape,
node distance=0.7cm,
auto,
main node/.style={fill,circle,draw,inner sep=0pt,minimum size=4pt}
]
\node [main node](1) {};
\node (2) [ below left of=1] {};
\node (3) [ below right of=1] {};
\node [main node](4) [below of =1]{};
\node [main node](5) [ left of=1, xshift=3mm,yshift=2.5mm] {};
\node [main node](6) [ right of=1, xshift=-3mm,yshift=2.5mm] {};
\node (7) [ above of=1] {};
\node [main node](8) [ above of=7] {};
\node [main node](9) [ below left of=2, xshift=-5mm] {};
\node [main node](10) [ below right of=3,xshift=5mm] {};
\node (11) [left of=5, xshift=-5mm] {$G_1=$};

\draw (1) -- (4);
\draw (1) -- (5);
\draw (1) -- (6);
\draw (1) -- (8);
\draw (1) -- (9);
\draw (1) -- (10);
\draw (4) -- (5);
\draw (4) -- (6);
\draw (5) -- (6);
\draw (4) -- (10);
\draw (5) -- (9);
\draw (6) -- (8);
\draw (8) -- (9);
\draw (8) -- (10);
\draw (9) -- (10);

\end{tikzpicture}

& 

\ \ \ \ \ \ \ \ \ \ \ \ \ \ \

&

\begin{tikzpicture}[
scale=0.8,
transform shape,
node distance=0.6cm,
auto,
main node/.style={fill,circle,draw,inner sep=0pt,minimum size=4pt}
]
\node [main node](1) {};
\node (2) [ below left of=1] {};
\node [main node](3) [ left of=2] {};
\node (4) [ below right of=1] {};
\node [main node](5) [ right of=4] {};
\node [main node](6) [ below of=1, yshift=-5mm] {};
\node [main node](7) [ below of=6] {};
\node (8) [ left of=7] {};
\node [main node](9) [ below left of=8] {};
\node (10) [ right of=7] {};
\node [main node](11) [ below right of=10] {};
\node (12) [below left of=3, xshift=-5mm] {$G_2=$};

\draw (1) -- (3);
\draw (1) -- (5);
\draw (1) -- (6);
\draw (3) -- (5);
\draw (3) -- (6);
\draw (5) -- (6);
\draw (7) -- (9);
\draw (7) -- (11);
\draw (9) -- (11);
\draw (3) -- (7);
\draw (3) -- (9);
\draw (5) -- (7);
\draw (5) -- (11);
\draw (6) -- (9);
\draw (6) -- (11);

\end{tikzpicture}
\end{tabular}

\caption{\label{nonRepTri1} The minimal (by the number of vertices) non-word-representable co-bipartite graphs $G_1$\cite{kitaev2015words} and $G_2$ \cite{kitaev2024human}.}
\end{center}
\end{figure}
A graph is said to be $k$-word-representable if there exists a $k$-uniform word representing it. According to Corollary~\ref{npc} and Proposition~\ref{knpc}, determining whether a graph is word-representable, as well as determining whether a graph is $k$-word-representable, are both NP-complete problems. From a $k$-word-representation of a word-representable graph, one can derive an upper bound on the representation number of the graph. Therefore, it is important to study the word-representability of specific graph classes and to obtain polynomial-time bounds on their representation numbers.

Split graphs and co-bipartite graphs play an important role in the study of the speed of hereditary graph classes and their asymptotic structure, and these aspects have been extensively studied in the literature. These concepts are also closely related to the asymptotic enumeration of word-representable graphs, as discussed in~\cite{collins2017new}. In general, for every infinite hereditary class $X$ of simple graphs, the following holds:
\[
\lim_{n\to\infty}\frac{\log_2 X_n}{\binom{n}{2}} = 1 - \frac{1}{k(X)}
\]
, where $k(X)$ is called the index of $X$. This index is defined using the graph classes ${\cal E}_{i,j}$, consisting of graphs whose vertices can be partitioned into at most $i$ independent sets and $j$ cliques. In particular, bipartite graphs, split graphs, and co-bipartite graphs correspond to the classes ${\cal E}_{2,0}$, ${\cal E}_{1,1}$, and ${\cal E}_{0,2}$, respectively. The index $k(X)$ is the maximum integer $k$ such that $X$ contains the class ${\mathcal E}_{i,j}$ for some integers $i$ and $j$ satisfying $i + j = k$. This result was independently established by Alekseev~\cite{alekseev1992range} and by Bollob\'as and Thomason~\cite{bollobas1995projections,bollobas1997hereditary}, and is now known as the Alekseev--Bollob\'as--Thomason Theorem (see also~\cite{alon2011structure}).

Since co-bipartite graphs are complements of bipartite graphs, they can be recognised in polynomial time. Furthermore, several computational problems that are difficult for general graphs admit polynomial-time algorithms when restricted to co-bipartite graphs. For example, Brandst\"adt et al.~\cite{brandstadt1998complexity} gave a polynomial-time algorithm to determine whether a graph contains a stable set whose removal results in a co-bipartite graph. Similarly, the $cd$-coloring problem can be solved in polynomial time for co-bipartite graphs~\cite{shalu2020complexity}. However, not all computational problems become easier on this class. Bodlaender and Jansen~\cite{bodlaender2000complexity} showed that the SIMPLE MAX CUT problem remains NP-complete for co-bipartite graphs. Moreover, the class of co-bipartite graphs has unbounded clique-width and contains at least $2^{n^2/4}$ distinct graphs on $n$ vertices~\cite{boliac2002clique}.

The study of word-representability of co-bipartite graphs is natural in the context of graph classes defined by vertex partitions. Bipartite graphs, where the vertex set can be partitioned into two independent sets, and split graphs, where the vertex set can be partitioned into one clique and one independent set, have already been studied in the context of word-representability. Co-bipartite graphs, where the vertex set can be partitioned into two cliques, form the complementary class of bipartite graphs. Therefore, it is natural to study the word-representability of co-bipartite graphs.

In this paper, we study word-representable co-bipartite graphs using vertex ordering. Vertex ordering provides a natural structural framework to describe the neighbourhood relationships between vertices. Using this ordering, we analyse the representation number, speed, and entropy of the word-representable co-bipartite graph. Moreover, this ordering enables the construction of word-representations and helps characterise the possible neighbourhood structures. Thus, vertex ordering serves as an effective tool for studying the structural and combinatorial properties of word-representable co-bipartite graphs.

The main goal of this paper is to study the word-representability of co-bipartite graphs via vertex ordering. In Section~\ref{sc2}, we introduce the necessary definitions and preliminaries related to word-representable graphs. In Section~\ref{sc3}, we study co-bipartite graphs based on a vertex ordering and derive a necessary and sufficient condition for a co-bipartite graph to be word-representable in terms of a specific ordering of the vertices in one of the cliques. In Section~\ref{sc4}, based on this ordering, we present an algorithm that constructs a $3$-uniform word-representation for any word-representable co-bipartite graph. As a consequence, we show that the representation number of any word-representable co-bipartite graph is at most~$3$. Moreover, we prove that the representation number of a co-bipartite graph is equal to~$3$ if and only if the graph is not a permutation graph. In Section \ref{sc5}, we study the speed and entropy of the class of word-representable co-bipartite graphs, and show that its speed is $2^{O(n \log n)}$ and its entropy is $0$.
\section{Preliminaries}\label{sc2}
In this section, we briefly describe the necessary definitions, notation, and basic results on word-representable graphs.

 A simple graph $G(V, E)$ is an undirected graph with a vertex set $V$ and an edge set $E$. For a vertex $v \in V$, $N_v$ denotes the set of vertices in $V$ that are adjacent (neighbours) to $v$. The \textit{degree} of a vertex $v$ is defined as the cardinality of $N_v$. A subset $S \subseteq V$ is called an \textit{independent set} if no two distinct vertices $u, v \in S$ are adjacent. A subset $C \subseteq V$ is called a \textit{clique} if every pair of distinct vertices $u, v \in C$ is adjacent. A clique $C$ is said to be \textit{maximal} if there is no vertex $v \in V \setminus C$ such that $C \cup \{v\}$ is also a clique in $G$.
 
A \textit{subword} of a word $w$ is a word obtained by removing certain letters from $w$. In a word $w$, if $x$ and $y$ alternate, then $w$ contains $xyxy\cdots$ or $yxyx\cdots$ (odd or even length) as a subword. 

\begin{dnt}\textit{(\cite{kitaev2015words} , Definition 3.0.5)}.
 A simple graph $G(V, E)$ is \textit{word-representable} if there exists a word $w$ over the alphabet $V$ such that letters $x$ and $y$, $x\neq y$, $x,y\in V$, alternate in $w$ if and only if $x$ and $y$ are adjacent in $G$. If a word $w$ \textit{represents} $G$, then $w$ contains each letter of $V$ at least once.
\end{dnt}

 In a word $w$ that represents a graph $G$, if $x$ and $y$ are not adjacent in $G$, then the non-alternation between $x$ and $y$ occurs if any one of these $xxy$, $yxx$, $xyy$, ${yyx}$ subwords is present in $w$.
\begin{dnt}
    (\textit{\cite{kitaev2015words}, Definition 3.2.1.}) A \textit{$k$-uniform word} is a word $w$ in which every letter occurs exactly $k$ times.
\end{dnt}

\begin{dnt}(\textit{\cite{kitaev2015words}, Definition 3.2.3.}) A graph is \textit{$k$-word-representable} (or \textit{$k$-representable}) if there exists a $k$-uniform word representing it.
\end{dnt}

\begin{dnt} (\textit{\cite{kitaev2017comprehensive}, Definition 3})
    For a word-representable graph $G$, \textit{the representation number} is the least $k$ such that $G$ is $k$-representable.
\end{dnt}
For a circle graph, the representation number is already known.
\begin{theorem}[\cite{kitaev2017comprehensive}, Theorem 6]\label{2unf}
		$\mathcal{R}_2 = \{G: G \text{ is a non-complete circle graph}\}$.
	\end{theorem}

\begin{prop}(\textit{\cite{kitaev2015words}, Proposition 3.0.15.})\label{pr2}
   Let $w = w_1xw_2xw_3$ be a word representing a graph $G$,
where $w_1$, $w_2$ and $w_3$ are possibly empty words, and $w_2$ contains no $x$. Let $X$ be the set of all letters that appear only once in $w_2$. Then, possible candidates for $x$ to be adjacent in $G$ are the letters in $X$. 
\end{prop}

 The \textit{initial permutation} of $w$ is the permutation obtained by removing all but the leftmost occurrence of each letter $x$ in $w$, and it is denoted by $\pi(w)$. Similarly, the \textit{final permutation} of $w$ is the permutation obtained by removing all but the rightmost occurrence of each letter $x$ in $w$, and it is denoted $\sigma(w)$. For a word $w$, $w_{\{x_1, \cdots, x_m\}}$ denotes the word after removing all letters except the letters $x_1, \ldots, x_m$ present in $w$. 
  \begin{exm}
      $w = 6345123215$, we have $\pi(w) = 634512$, $\sigma(w) = 643215$ and $w_{\{6,5\}} = 655$.
   \end{exm}

There exists a connection between graph orientations and word-representability. In the following, the definition of a semi-transitive orientation and its connection to word-representability are described.

Semi-transitive orientation is defined based on shortcuts in the papers \cite{halldorsson2011alternation} and \cite{halldorsson2016semi}.
 A semi-cycle is the directed acyclic graph obtained by reversing the
direction of one edge of a directed cycle \cite{kitaev2015words}. An acyclic digraph is a {\em shortcut} if it is induced by the vertices of a semi-cycle and contains a pair of non-adjacent vertices. Thus, any shortcut  

\begin{itemize}

\item is {\em acyclic} (that is, there are {\em no directed cycles});

\item has {\em at least} 4 vertices;

\item has {\em exactly one} source (the vertex with no edges coming in), {\em exactly one} sink (the vertex with no edges coming out), and a {\em directed path} from the source to the sink that goes through {\em every} vertex in the graph;

\item has an edge connecting the source to the sink that we refer to as the {\em shortcutting edge};

\item is {\em not} transitive (that is, there exist vertices $u$, $v$ and $z$ such that $u\rightarrow v$ and $v\rightarrow z$ are edges, but there is {\em no} edge $u\rightarrow z$).

\end{itemize}

\begin{dnt}(\cite{halldorsson2016semi}, Definition 1) An orientation of a graph is {\em semi-transitive} if it is {\em acyclic} and 
{\em shortcut-free}.\end{dnt}

From these definitions, it is clear that {\em any} transitive orientation is necessarily semi-transitive. The converse is {\em not} true. Thus, semi-transitive orientations generalise transitive orientations.
A key result in the theory of word-representable graphs is the following theorem. 

\begin{theorem}(\cite{halldorsson2016semi}, Theorem 3)\label{key-thm} A graph is word-representable if and only if it admits a semi-transitive orientation. Moreover, each non-complete word-representable graph is $2(n-\kappa)$-word-representable where $\kappa$ is the size of the maximum clique in~$G$.
		\label{thm:rep-equals-semi-trans}
	\end{theorem}
	
    The recognition problem of word-representable graphs and deciding the representation number of word-representable graphs are $NP$-complete problems.
	\begin{cor}(\cite{halldorsson2016semi}, Corollary 2)\label{npc}
		The recognition problem for word-representable graphs is NP-complete.
	\end{cor}
    \begin{prop}(\cite{halldorsson2016semi},  Proposition 8.)\label{knpc}
        Deciding whether a given graph is a $k$-word-representable graph, for any given $3 \leq k \leq \lceil n/2\rceil$, is NP-complete. 
    \end{prop}
The following characterisation of word-representable split graphs is derived using vertex labelling. In what follows, for any two integers $a \le b$, we denote the set of integers $\{a, a+1, \ldots, b\}$ by $[a, b]$.

\begin{theorem}[\cite{kitaev2024semi}]\label{Word_split_graph}
	Let $G (I \cup C, E)$ be a split graph. Then, $G$ is word-representable if and only if the vertices of $C$ can be labelled from $1$ to $k = |C|$ in such a way that for each $a, b \in I$ the following holds.
	\begin{enumerate}
		\item  Either $N(a) = [1, m] \cup [n, k]$, for $m < n$, or $N(a) = [l, r]$, for $l \le r$.
		\item  If $N(a) = [1, m] \cup [n, k]$ and $N(b) = [l, r]$, for $m < n$ and $l \le r$, then $l > m$ or $r < n$.
		\item If $N(a) = [1, m] \cup [n, k]$ and $N(b) = [1, m'] \cup [n', k]$, for $m < n$ and $m' < n'$, then $m' < n$ and $m < n'$.
	\end{enumerate}
\end{theorem}
A word $u$ contains a word $v$ as a \textit{factor} if $u = xvy$ where $x$ and $y$ can be empty words. In this paper, $w=w_1w_2\cdots w_n$ denotes the word $w$ contains $\{w_1,w_2,\ldots, w_n\}$ as factors where $w_i$ is a word possibly empty. In a graph $G(V,E)$, we denote the edges between two vertices $x$ and $y$ as $x \sim y$. If there is no edge between them, we denote this as $x \nsim y$.
\noindent 

\section{Necessary and sufficient condition in vertex ordering}\label{sc3}
The word-representability of co-bipartite graphs is studied in \cite{das2025word}, where necessary and sufficient conditions for semi-transitive orientations of co-bipartite graphs are established. The known results concerning the semi-transitive orientation of co-bipartite graphs are as follows.

\begin{obs}(\textit{\cite{das2025word}})\label{obs41}
    Any semi-transitive orientation of $\overline{B}(K_m,K_n)$ subdivides the set of all vertices into three possibly empty groups of the types shown schematically in Figure~\ref{3-groups}. 
\begin{itemize}
\item %Similar to Figure~\ref{schem-structure}, 
The vertical oriented paths are a schematic way to show (parts of) $\vec{P_1}$ for $V(K_m)$ and $\vec{P_2}$ for $V(K_n)$;
\item For the vertices of $K_m$, the vertical oriented paths in the types $A$ and $B$ represent up to $l$ consecutive vertices in $\vec{P_1}$, $l\leq n$. Similarly, for the vertices of $K_n$, the vertical oriented paths in the types $A$ and $B$ represent up to $l$ consecutive vertices in $\vec{P_2}$, $l\leq m$.

\begin{figure}[H]
\begin{center}

\begin{tabular}{ccc}

\begin{tikzpicture}[
->,
>=stealth',
shorten >=1pt,
scale=0.8,
transform shape,
node distance=0.5cm,
auto,
main node/.style={fill,circle,draw,inner sep=0pt,minimum size=4pt}
]
\node[main node] (1) {};
\node[main node] (2) [right of=1,xshift=5mm] {};
\node[main node] (3) [above of=2] {};
\node (4) [above of=3] {};
\node[main node] (5) [below of=2] {};
\node (6) [below of=5] {};

\node (7) [above of=4,yshift=-4mm,xshift=1mm] {};
\node (8) [above right of=7,xshift=6mm]{};

\node (10) [below of=6,yshift=4mm,xshift=1mm] {};
\node (11) [below right of=10,xshift=6mm]{};

\node (11) [below left of=5,yshift=-5mm]{type $A$};

\path
(4) edge (3)
(3) edge (2)
(2) edge (5)
(5) edge (6);

\path
(1) edge (2)
(1) edge (3)
(1) edge (5);

\end{tikzpicture}

&

\begin{tikzpicture}[
->,
>=stealth',
shorten >=1pt,
scale=1,
transform shape,
node distance=0.5cm,
auto,
main node/.style={fill,circle,draw,inner sep=0pt,minimum size=4pt}
]
\node[main node] (1) {};
\node[main node] (2) [right of=1,xshift=5mm] {};
\node[main node] (3) [above of=2] {};
\node (4) [above of=3] {};
\node[main node] (5) [below of=2] {};
\node (6) [below of=5] {};

\node (7) [above of=4,yshift=-4mm,xshift=1mm] {};
\node (8) [above right of=7,xshift=6mm]{};

\node (10) [below of=6,yshift=4mm,xshift=1mm] {};
\node (11) [below right of=10,xshift=6mm]{};

\node (11) [below left of=5,yshift=-5mm]{type $B$};

\path
(4) edge (3)
(3) edge (2)
(2) edge (5)
(5) edge (6);

\path
(2) edge (1)
(3) edge (1)
(5) edge (1);

\end{tikzpicture}

&

\begin{tikzpicture}[
->,
>=stealth',
shorten >=1pt,
scale=0.8,
transform shape,
node distance=0.5cm,
auto,
main node/.style={fill,circle,draw,inner sep=0pt,minimum size=4pt}
]
\node[main node] (1) {};
\node[main node] (2) [below right of=1,xshift=5mm] {};
%\node[main node] (3) [above of=2] {};
\node[main node] (4) [above of=2] {};
\node[main node] (5) [below of=2] {};
\node[main node] (6) [below of=5] {};
\node[main node] (13) [above of=4] {};
\node[main node] (14) [below of=5] {};

\node (7) [above of=13,yshift=-5mm,xshift=1mm] {};
\node (8) [right of=7,xshift=6mm]{};
\node (9) [right of=8]{source};

\node (10) [below of=14,yshift=5mm,xshift=1mm] {};
\node (11) [right of=10,xshift=6mm]{};
\node (12) [right of=11,xshift=-1mm]{sink};

\node [below of=14]{type $C$};

\path
(8) edge (7)
(11) edge (10);

\path
(13) edge (1)
(4) edge (1)
(1) edge (14)
(1) edge (5)
(4) edge (2)
(2) edge (5)
(13) edge (4)
(5) edge (14);
\end{tikzpicture}
\end{tabular}

\caption{\label{3-groups} Three types of vertices in $K_m$ and $K_n$ }
\end{center}
\end{figure}
\item For the vertices of $K_m$, the vertical oriented path in type $C$ contains all $n$ vertices in $\vec{P_1}$. Similarly, for the vertices of $K_n$, the vertical oriented path in type $C$ contains all $m$ vertices in $\vec{P_2}$. These oriented paths $P_1$ and $P_2$ are subdivided into three groups of consecutive vertices, with the middle group possibly containing no vertices; the group of vertices containing the source (resp., sink) is the {\em source-group} (resp., {\em sink-group});
\end{itemize} 
\end{obs}

\begin{lemma}(\textit{\cite{das2025word}})\label{lm41}
    Suppose that the graph $\overline{B}(K_m,K_n)$ is word-representable and $S$ is a semi-transitive orientation of the graph $\overline{B}(K_m,K_n)$. Suppose $\{x,y\}\in V(K_m)$ and $x$ is a type $A$ vertex and $y$ is a type $B$ vertex and the orientation of the edge between $x$ and $y$ is $y\rightarrow x$. Then, $x$ and $y$ do not have any common neighbour among the vertices of $K_n$. 
\end{lemma}

\begin{lemma}(\textit{\cite{das2025word}}) \label{lm42}
    Suppose the graph $\overline{B}(K_m,K_n)$ is word-representable and $S$ is a semi-transitive orientation of the graph $\overline{B}(K_m,K_n)$. Suppose $\{x,y\}\in V(K_m)$ and $\{x_s,x_{s+1}\}\in V(K_n)$, and the orientation of the edges between $\{x,y\}$ is $x \rightarrow y$, and between $\{x_s,x_{s+1}\}$ is $x_s \rightarrow x_{s+1}$ in the semi-transitive orientation $S$. Then, the semi-transitive orientation $S$ follows the following conditions:
    \begin{enumerate}
        \item If the orientation of the edges between $\{x_s,x\}$ is $x_s\rightarrow x$ and between $\{y,x_{s+1}\}$ is $y\rightarrow x_{s+1}$ in the semi-transitive orientation $S$, then $x\rightarrow x_{s+1}$ and $x_s\rightarrow y$ in the semi-transitive orientation $S$.
        \item If the orientation of the edges between $\{x,x_{s+1}\}$ is $x\rightarrow x_{s+1}$ and between $\{x_s,x\}$ is $y\rightarrow x_s$ in the semi-transitive orientation $S$, then $x\rightarrow x_s$ and $y\rightarrow x_{s+1}$ in the semi-transitive orientation $S$.
    \end{enumerate}
\end{lemma}

\begin{lemma}(\textit{\cite{das2025word}})\label{lm43}
Suppose the graph $\overline{B}(K_m,K_n)$ is word-representable and $S$ is a semi-transitive orientation of the graph. Let $x\in V(\overline{B}(K_m,K_n))$, and $x$ is a type $C$ vertex. Suppose $x_s$ and $x_{s+1}$ are the last and first vertices of the source and sink group, respectively.
Now, for $y\in V(\overline{B}(K_m,K_n))$, if $x$ and $y$ are in the same partition in the graph $V(\overline{B}(K_m,K_n))$, then the following conditions hold for $y$.

\begin{enumerate}
    \item If the edges between $x$ and $y$ are oriented as $x\rightarrow y$, $y$ is not a type $A$ vertex that is adjacent to both $x_s$ and $x_{s+1}$. Also, if $y$ is a type $C$ vertex, then $x_s$ cannot be contained in the sink group of $y$.
    \item If the edges between $x$ and $y$ are oriented as $y\rightarrow x$, $y$ is not a type $B$ vertex that is adjacent to both $x_s$ and $x_{s+1}$. Also, if $y$ is a type $C$ vertex, then $x_{s+1}$ cannot be contained in the source group of $y$.
\end{enumerate}
\end{lemma}

\begin{theorem}(\textit{\cite{das2025word}})\label{thm45}
    A co-bipartite graph $\overline{B}(K_m,K_n)$ has a semi-transitive orientation if and only if 
\begin{itemize} 
\item $K_m$ and $K_n$ are oriented transitively,
\item In the graph $\overline{B}(K_m,K_n)$, every vertex belongs to one of the three types described in Observation \ref{obs41}, and 
\item The conditions present in Lemmas \ref{lm41}, \ref{lm42} and \ref{lm43} are satisfied. 
\end{itemize}
\end{theorem}

We study the conditions under which a vertex ordering of a co-bipartite graph becomes word-representable. We derive a necessary and sufficient condition for a word-representable co-bipartite graph based on the vertex ordering obtained from Observation \ref{obs41}. In the following theorems, we denote the neighbours of a type $A$ or $B$ vertex $a_i$ (excluding the vertices of $K_m$) as $[x_i,y_i]=\{x_i,x_{i+1},\ldots, y_{i-1},y_i\}$, where $\{x_i,x_{i+1},\ldots, y_{i-1},y_i\}\subseteq V(K_n)$ and the neighbour of type $C$ vertex $c_i$ (excluding the vertices of $K_m$) as $[1,x_i]\cup[y_i,n]=\{1,2,\ldots,x_i,y_i,\ldots,n\}$, where $\{1,2,\ldots,x_i,y_i,\ldots,n\}\subseteq V(K_n)$.

\begin{theorem}\label{thmlab}
    A co-bipartite graph $\overline{B}(K_m,K_n)$, where $V(K_n)=\{1,2,\ldots,n\}$, is word-representable if and only if there exists a vertex ordering for the vertices of $K_m$ that satisfies the following conditions:
    \begin{enumerate}
      %  \item The rows and columns of $M_{m\times n}$ has circular one property.
       \item The ordering of type $A$, $B$ and $C$ vertices is $A<C<B$. 
       %(From Lemma 3.3,3.5(1,2))
        \item Suppose $c_i$ and $c_j$ are type $C$ vertices, $c_i<c_j$, $c_i,c_j\in V(K_m)$. Let $N_{c_i}\setminus V(K_m)=[1,x_i]\cup [y_i,n]$, $N_{c_j}\setminus V(K_m)=[1,x_j]\cup [y_j,n]$. Then $x_i\leq x_j$ and $y_i\leq y_j$. 
        %(From Lemma 3.5(1,2) and Lemma 3.4(1))
        \item Suppose $a_i$ and  $a_j$ are type $A$ or type $B$ vertices, $a_i<a_j$, $a_i,a_j\in V(K_m)$. Let $N_{a_i}\setminus V(K_m)=[x_i,y_i]$, $N_{a_j}\setminus V(K_m)=[x_j,y_j]$. Then $x_i\leq x_j$ and $y_i\leq y_j$. 
        %(From Lemma 3.4(2)) 
        \item Suppose $a_i$ is type $A$ (or $B$) vertex and $c_j$ is type $C$ vertex, $a_i<c_j$(or $a_i>c_j$ for type $B$), $a_i,c_j\in V(K_m)$. 
        Let $N_{a_i}\setminus V(K_m)=[x_i,y_i]$, $N_{c_j}\setminus V(K_m)=[1,x_j]\cup [y_j,n]$. Then $x_i\leq y_j$ and $x_j\leq y_i$.
        %(From Lemma 3.4(2))
        \item Suppose $a_i$ is type $A$ vertex and $b_j$ is type $B$ vertex. Let $N_{a_i}\setminus V(K_m)=[x_i,y_i]$, $N_{b_j}\setminus V(K_m)=[x_j,y_j]$. Then $x_j\leq x_i$ and $y_j\leq y_i$. 
        %(From Lemma 3.4(1))
        \item Suppose $v\in V(K_m)$ and $N_v\setminus V(K_m)=\emptyset$. Let $a_i$ be the last type $A$ vertex and $c_j$ be any type $C$ vertex. Then $a_i<v<c_j$ and $y_i<y_j$, where $N_{a_i}\setminus V(K_m)=[x_i,y_i]$, $N_{c_j}\setminus V(K_m)=[1,x_j]\cup [y_j,n]$.
    \end{enumerate}
\end{theorem}

\begin{proof}
We prove necessity by showing that violating any condition creates a shortcut, and sufficiency by showing that no shortcut can arise when all conditions are satisfied. Suppose there exists a semi-transitive orientation $S$ of the word-representable graph $\overline{B}(K_m,K_n)$ such that these conditions are not satisfied by $S$. In the following cases, we show that each condition is necessary for the orientation $S$.
    
     \textbf{Case 1}: Suppose the ordering of types $A$, $B$, and $C$ is not $A < C < B$. If either $B < A$ or $C < A$, then a shortcut occurs where $u$ is a type $B$ or type $C$ vertex and $v$ is a type $A$ vertex, since the edge $v \rightarrow x_s$ cannot exist. Here, $x_s$ and $x_{s+1}$ are vertices of $K_n$, and the induced subgraph on $\{u, v, x_s, x_{s+1}\}$ forms a shortcut, as shown in Figure~\ref{Case 1}.

    \begin{figure}[H]
\begin{center}
\begin{tikzpicture}[node distance=1cm,auto,main node/.style={fill,circle,draw,inner sep=1pt,minimum size=5pt}]

\node [main node](1) {};
\node [right of=1,xshift=-5mm] {$x_{s}$};
\node [main node](2) [ below of=1] {};
\node [right of=2,xshift=-5mm] {$x_{s+1}$};
\node [main node](3) [ left of=1] {};
\node [left of=3,xshift=5mm] {$u$};
\node [main node](4) [ left of=2] {};
\node [left of=4,xshift=5mm] {$v$};

\draw[->] (1) -- (2);
\draw[->] (3) -- (4);
\draw[->] (1) -- (3);
\draw[->] (4) -- (2);
%\draw [->] (1) to [in=750, distance=1.5cm] (3);

\end{tikzpicture}
\caption{\label{Case 1} Induced subgraph of the vertices $B$, $A$, $x_{s}$ and $x_{s+1}$.}
\end{center}
\end{figure}
Also, if $B < C$, the same shortcut occurs, where $u$ is a type $B$ vertex and $v$ is a type $C$ vertex, since the edge $v \rightarrow x_s$ cannot exist. Therefore, the ordering of types $A$, $B$, and $C$ must be $A < C < B$.

 \textbf{Case 2:} Suppose $c_i$ and $c_j$ are two type $C$ vertices with $c_i<c_j$ such that $x_i>x_j$ or $y_i>y_j$. For both of these cases, we obtain the following shortcuts: $x_j\rightarrow x_i\rightarrow c_i\rightarrow c_j$ and $c_i\rightarrow c_j\rightarrow y_j\rightarrow y_i$, respectively. Therefore, $x_i\leq x_j$ and $y_i\leq y_j$.
\begin{figure}[H]
\begin{center}

\begin{tabular}{ccc}

\begin{tikzpicture}[node distance=1cm,auto,main node/.style={fill,circle,draw,inner sep=1pt,minimum size=5pt}]

\node [main node](1) {};
\node [right of=1,xshift=-5mm] {$x_{j}$};
\node [main node](2) [ below of=1] {};
\node [right of=2,xshift=-5mm] {$x_{i}$};
\node [main node](3) [ left of=1] {};
\node [left of=3,xshift=5mm] {$c_i$};
\node [main node](4) [ left of=2] {};
\node [left of=4,xshift=5mm] {$c_j$};
\node [below right of=4,xshift=-2.5mm] {(I)};
\draw[->] (1) -- (2);
\draw[->] (3) -- (4);
%\draw[->] (1) -- (3);
\draw[->] (1) -- (4);
\draw[->] (2) -- (3);
%\draw [->] (1) to [in=750, distance=1.5cm] (3);

\end{tikzpicture}

\ \ \ \ \  \ \ \ \ \ \ \ \ \ \ \ \ \  \ \ &
\begin{tikzpicture}[node distance=1cm,auto,main node/.style={fill,circle,draw,inner sep=1pt,minimum size=5pt}]

\node [main node](1) {};
\node [right of=1,xshift=-5mm] {$y_{j}$};
\node [main node](2) [ below of=1] {};
\node [right of=2,xshift=-5mm] {$y_{i}$};
\node [main node](3) [ left of=1] {};
\node [left of=3,xshift=5mm] {$c_i$};
\node [main node](4) [ left of=2] {};
\node [left of=4,xshift=5mm] {$c_j$};
\node [below right of=4,xshift=-2.5mm] {(II)};
\draw[->] (1) -- (2);
\draw[->] (3) -- (4);
\draw[->] (4) -- (1);
%\draw[->] (4) -- (2);
\draw[->] (3) -- (2);
%\draw [->] (1) to [in=750, distance=1.5cm] (3);

\end{tikzpicture}

\end{tabular}

\caption{\label{Case 2} Induced subgraph of the vertices $c_i$, $c_j$, $x_{i}$, $x_{j}$ (I) and $c_i$, $c_j$, $y_{i}$, $y_{j}$ (II).}
\end{center}
\end{figure}

 \textbf{Case 3:} Suppose $a_i$ and $a_j$ are type $A$ vertices, $a_i<a_j$ such that $x_i>x_j$ or $y_i>y_j$. For both of these cases, we obtain the following shortcuts: $a_i\rightarrow a_j\rightarrow x_j\rightarrow x_i$ and $a_i\rightarrow a_j\rightarrow y_j\rightarrow y_i$, respectively, as shown in Figure \ref{Case 3a}. Therefore, $x_i\leq x_j$ and $y_i\leq y_j$.
\begin{figure}[H]
\begin{center}
\begin{tabular}{ccc}
\begin{tikzpicture}[node distance=1cm,auto,main node/.style={fill,circle,draw,inner sep=1pt,minimum size=5pt}]

\node [main node](1) {};
\node [right of=1,xshift=-5mm] {$x_{j}$};
\node [main node](2) [ below of=1] {};
\node [right of=2,xshift=-5mm] {$x_{i}$};
\node [main node](3) [ left of=1] {};
\node [left of=3,xshift=5mm] {$a_i$};
\node [main node](4) [ left of=2] {};
\node [left of=4,xshift=5mm] {$a_j$};
\node [below right of=4,xshift=-2.5mm] {(I)};
\draw[->] (1) -- (2);
\draw[->] (3) -- (4);
\draw[->] (3) -- (2);
\draw[->] (4) -- (1);
%\draw[->] (2) -- (3);
%\draw [->] (1) to [in=750, distance=1.5cm] (3);
\end{tikzpicture}
\ \ \ \ \  \ \ \ \ \ \ \ \ \ \ \ \ \  \ \ &
\begin{tikzpicture}[node distance=1cm,auto,main node/.style={fill,circle,draw,inner sep=1pt,minimum size=5pt}]

\node [main node](1) {};
\node [right of=1,xshift=-5mm] {$y_{j}$};
\node [main node](2) [ below of=1] {};
\node [right of=2,xshift=-5mm] {$y_{i}$};
\node [main node](3) [ left of=1] {};
\node [left of=3,xshift=5mm] {$a_i$};
\node [main node](4) [ left of=2] {};
\node [left of=4,xshift=5mm] {$a_j$};
\node [below right of=4,xshift=-2.5mm] {(II)};
\draw[->] (1) -- (2);
\draw[->] (3) -- (4);
\draw[->] (4) -- (1);
%\draw[->] (4) -- (2);
\draw[->] (3) -- (2);
%\draw [->] (1) to [in=750, distance=1.5cm] (3);
\end{tikzpicture}
\end{tabular}
\caption{\label{Case 3a} Induced subgraph of the vertices $a_i$, $a_j$, $x_{i}$, $x_{j}$ (I) and $a_i$, $a_j$, $y_{i}$, $y_{j}$ (II).}
\end{center}
\end{figure}

Suppose $b_i$ and $b_j$ are type $B$ vertices, $b_i<b_j$ such that $x_i>x_j$ or $y_i>y_j$. For both of these cases, we obtain the following shortcuts: $x_j\rightarrow x_i\rightarrow b_i\rightarrow b_j$ and $y_j\rightarrow y_i\rightarrow b_i\rightarrow b_j$, respectively, as shown in Figure \ref{Case 3b}. Therefore, $x_i\leq x_j$ and $y_i\leq y_j$.
\begin{figure}[H]
\begin{center}

\begin{tabular}{ccc}

\begin{tikzpicture}[node distance=1cm,auto,main node/.style={fill,circle,draw,inner sep=1pt,minimum size=5pt}]

\node [main node](1) {};
\node [right of=1,xshift=-5mm] {$x_{j}$};
\node [main node](2) [ below of=1] {};
\node [right of=2,xshift=-5mm] {$x_{i}$};
\node [main node](3) [ left of=1] {};
\node [left of=3,xshift=5mm] {$b_i$};
\node [main node](4) [ left of=2] {};
\node [left of=4,xshift=5mm] {$b_j$};
\node [below right of=4,xshift=-2.5mm] {(I)};
\draw[->] (1) -- (2);
\draw[->] (3) -- (4);
\draw[->] (2) -- (3);
\draw[->] (1) -- (4);
%\draw[->] (2) -- (3);
%\draw [->] (1) to [in=750, distance=1.5cm] (3);

\end{tikzpicture}

\ \ \ \ \  \ \ \ \ \ \ \ \ \ \ \ \ \  \ \ &
\begin{tikzpicture}[node distance=1cm,auto,main node/.style={fill,circle,draw,inner sep=1pt,minimum size=5pt}]

\node [main node](1) {};
\node [right of=1,xshift=-5mm] {$y_{j}$};
\node [main node](2) [ below of=1] {};
\node [right of=2,xshift=-5mm] {$y_{i}$};
\node [main node](3) [ left of=1] {};
\node [left of=3,xshift=5mm] {$a_i$};
\node [main node](4) [ left of=2] {};
\node [left of=4,xshift=5mm] {$a_j$};
\node [below right of=4,xshift=-2.5mm] {(II)};
\draw[->] (1) -- (2);
\draw[->] (3) -- (4);
\draw[->] (1) -- (4);
%\draw[->] (4) -- (2);
\draw[->] (2) -- (3);
%\draw [->] (1) to [in=750, distance=1.5cm] (3);

\end{tikzpicture}

\end{tabular}

\caption{\label{Case 3b} Induced subgraph of the vertices $b_i$, $b_j$, $x_{i}$, $x_{j}$ (I) and $a_i$, $a_j$, $y_{i}$, $y_{j}$ (II).}
\end{center}
\end{figure}

 \textbf{Case 4:} Suppose $a_i$, $c_j$ are type $A$ and type $C$ vertices, respectively, $a_i<c_j$ such that $x_i>y_j$ or $x_j>y_i$. For both of these cases, we obtain the shortcuts: $a_i\rightarrow c_j\rightarrow y_j\rightarrow x_i$ and $a_i\rightarrow y_i\rightarrow x_j\rightarrow c_j$, respectively. Hence,  $x_i\leq y_j$ and $x_j\leq y_i$. 
 \begin{figure}[H]
\begin{center}

\begin{tabular}{ccc}

\begin{tikzpicture}[node distance=1cm,auto,main node/.style={fill,circle,draw,inner sep=1pt,minimum size=5pt}]
\node [main node](1) {};
\node [right of=1,xshift=-5mm] {$y_{j}$};
\node [main node](2) [ below of=1] {};
\node [right of=2,xshift=-5mm] {$x_{i}$};
\node [main node](3) [ left of=1] {};
\node [left of=3,xshift=5mm] {$a_i$};
\node [main node](4) [ left of=2] {};
\node [left of=4,xshift=5mm] {$c_j$};
\node [below right of=4,xshift=-2.5mm] {(I)};
\draw[->] (1) -- (2);
\draw[->] (3) -- (4);
%\draw[->] (1) -- (3);
\draw[->] (4) -- (1);
\draw[->] (3) -- (2);
%\draw [->] (1) to [in=750, distance=1.5cm] (3);

\end{tikzpicture}

\ \ \ \ \  \ \ \ \ \ \ \ \ \ \ \ \ \  \ \ &
\begin{tikzpicture}[node distance=1cm,auto,main node/.style={fill,circle,draw,inner sep=1pt,minimum size=5pt}]

\node [main node](1) {};
\node [right of=1,xshift=-5mm] {$y_{i}$};
\node [main node](2) [ below of=1] {};
\node [right of=2,xshift=-5mm] {$x_{j}$};
\node [main node](3) [ left of=1] {};
\node [left of=3,xshift=5mm] {$a_i$};
\node [main node](4) [ left of=2] {};
\node [left of=4,xshift=5mm] {$c_j$};
\node [below right of=4,xshift=-2.5mm] {(II)};
\draw[->] (1) -- (2);
\draw[->] (3) -- (4);
\draw[->] (3) -- (1);
%\draw[->] (4) -- (2);
\draw[->] (2) -- (4);
%\draw [->] (1) to [in=750, distance=1.5cm] (3);

\end{tikzpicture}

\end{tabular}

\caption{\label{Case 4} Induced subgraph of the vertices $a_i$, $c_j$, $x_{i}$, $y_{j}$ (I) and $a_i$, $c_j$, $y_{i}$, $x_{j}$ (II).}
\end{center}
\end{figure}

  \textbf{Case 5:} Suppose $a_i$, $b_j$ are type $A$ and type $B$ vertices, respectively, $a_i<b_j$ such that $x_i<x_j$ or $y_i<y_j$. For both of these cases, we obtain the shortcuts: $a_i\rightarrow x_i\rightarrow x_j\rightarrow b_j$ and  $a_i\rightarrow y_i\rightarrow y_j\rightarrow b_j$, respectively. Hence, $x_i\geq x_j$ and $y_i\geq y_j$.
\begin{figure}[H]
\begin{center}

\begin{tabular}{ccc}

\begin{tikzpicture}[node distance=1cm,auto,main node/.style={fill,circle,draw,inner sep=1pt,minimum size=5pt}]

\node [main node](1) {};
\node [right of=1,xshift=-5mm] {$x_{i}$};
\node [main node](2) [ below of=1] {};
\node [right of=2,xshift=-5mm] {$x_{j}$};
\node [main node](3) [ left of=1] {};
\node [left of=3,xshift=5mm] {$a_i$};
\node [main node](4) [ left of=2] {};
\node [left of=4,xshift=5mm] {$b_j$};
\node [below right of=4,xshift=-2.5mm] {(I)};
\draw[->] (1) -- (2);
\draw[->] (3) -- (4);
\draw[->] (3) -- (1);
\draw[->] (2) -- (4);
%\draw[->] (2) -- (3);
%\draw [->] (1) to [in=750, distance=1.5cm] (3);

\end{tikzpicture}

\ \ \ \ \  \ \ \ \ \ \ \ \ \ \ \ \ \  \ \ &
\begin{tikzpicture}[node distance=1cm,auto,main node/.style={fill,circle,draw,inner sep=1pt,minimum size=5pt}]

\node [main node](1) {};
\node [right of=1,xshift=-5mm] {$y_{i}$};
\node [main node,yshift=1mm](2) [ below of=1] {};
\node [right of=2,xshift=-5mm] {$y_{j}$};
\node [main node](3) [ left of=1] {};
\node [left of=3,xshift=5mm] {$a_i$};
\node [main node](4) [ left of=2] {};
\node [left of=4,xshift=5mm] {$b_j$};
\node [below right of=4,xshift=-2.5mm] {(II)};
\draw[->] (1) -- (2);
\draw[->] (3) -- (4);
\draw[->] (3) -- (1);
\draw[->] (2) -- (4);
%\draw [->] (1) to [in=750, distance=1.5cm] (3);

\end{tikzpicture}

\end{tabular}

\caption{\label{Case 5} Induced subgraph of the vertices $a_i$, $b_j$, $x_{i}$, $x_{j}$ (I) and $a_i$, $b_j$, $y_{i}$, $y_{j}$ (II).}
\end{center}
\end{figure}

\textbf{Case 6:} For a vertex $v$, $v\in V(K_m)$ and $N_v\setminus V(K_m)=\emptyset$, suppose $v<a_i$ or $c_j<v$, $a_i$ is the last type $A$ vertex and $c_j$ is any type $C$ vertex. Then the following shortcut occurs: $v\rightarrow a_i\rightarrow x_i\rightarrow c_j$ and $a_i\rightarrow y_i\rightarrow c_j\rightarrow v$, respectively. Hence, $v\geq a_i$ and $c_j\geq v$.

\begin{figure}[H]
\begin{center}

\begin{tabular}{ccc}

\begin{tikzpicture}[node distance=1cm,auto,main node/.style={fill,circle,draw,inner sep=1pt,minimum size=5pt}]

\node [main node](1) {};
\node [right of=1,xshift=-5mm] {$x_{i}$};
\node [main node ](2) [ below of=1] {};
\node [right of=2,xshift=-5mm] {$x_{j}$};
\node [main node](3) [ left of=1] {};
\node [left of=3,xshift=5mm] {$a_i$};
\node [main node](4) [ left of=2] {};
\node [left of=4,xshift=5mm] {$c_j$};
\node [below right of=4,xshift=-2.5mm] {(I)};
\node [main node](5) [ above of=3] {};
\node [right of=5,xshift=-5mm] {$v$};
\draw[->] (1) -- (2);
\draw[->] (3) -- (4);
\draw[->] (3) -- (1);
\draw[->] (2) -- (4);
\draw[->] (3) -- (2);
\draw[->] (1) -- (4);
\draw[->] (5) -- (3);
\draw [->] (5) to [in=160, out=200, distance=1.2cm] (4);
\end{tikzpicture}

\ \ \ \ \  \ \ \ \ \ \ \ \ \ \ \ \ \  \ \ &
\begin{tikzpicture}[node distance=1cm,auto,main node/.style={fill,circle,draw,inner sep=1pt,minimum size=5pt}]

\node [main node](1) {};
\node [right of=1,xshift=-5mm] {$y_{i}$};
\node [main node](2) [ below of=1] {};
\node [right of=2,xshift=-5mm] {$y_{j}$};
\node [main node](3) [ left of=1] {};
\node [left of=3,xshift=5mm] {$a_i$};
\node [main node](4) [ left of=2] {};
\node [left of=4,xshift=5mm] {$c_j$};
\node [main node](5) [ below of=4] {};
\node [right of=5,xshift=-5mm] {$v$};
\node [below right of=5,xshift=-2.5mm] {(II)};

\draw[->] (1) -- (2);
\draw[->] (3) -- (4);
\draw[->] (3) -- (1);
\draw[->] (2) -- (4);
\draw[->] (3) -- (2);
\draw[->] (1) -- (4);
\draw[->] (5) -- (3);
\draw [->] (3) to [in=160, out=200, distance=1.2cm] (5);
\end{tikzpicture}
\end{tabular}
\caption{\label{Case 61} Induced subgraph of the vertices $v$, $a_i$, $c_j$, $x_{i}$, $x_{j}$ (I) and $v$, $a_i$, $c_j$, $y_{i}$, $y_{j}$ (II).}
\end{center}
\end{figure}
Now, suppose $a_i<v<c_j$ and $y_i\geq y_j$, where $N_{a_i}\setminus V(K_m)=[x_i,y_i]$, $N_{c_j}\setminus V(K_m)=[1,x_j]\cup [y_j,n]$. Then, the shortcut $a_i\rightarrow v\rightarrow c_j\rightarrow y_j$ occurs. Therefore, $y_i< y_j$.
\begin{figure}[H]
\begin{center}

\begin{tikzpicture}[node distance=1cm,auto,main node/.style={fill,circle,draw,inner sep=1pt,minimum size=5pt}]

\node [main node](1) {};
\node [right of=1,xshift=-5mm] {$y_{j}$};
\node [main node](2) [ below of=1] {};
\node [right of=2,xshift=-5mm] {$y_{i}$};
\node [main node](3) [ left of=1] {};
\node [left of=3,xshift=5mm] {$v$};
\node [main node](4) [ left of=2] {};
\node [left of=4,xshift=5mm] {$c_j$};
%\node [below right of=4,xshift=-2.5mm] {(I)};
\node [main node](5) [ above of=3] {};
\node [right of=5,xshift=-5mm] {$a_i$};
\draw[->] (1) -- (2);
\draw[->] (5) -- (4);
\draw[->] (5) -- (1);
\draw[->] (4) -- (2);
\draw[->] (5) -- (2);
\draw[->] (4) -- (1);
\draw[->] (5) -- (3);
\draw [->] (5) to [in=160, out=200, distance=1.2cm] (4);

\end{tikzpicture}
\caption{\label{Case 62} Induced subgraph of the vertices $a_i$, $v$, $c_j$, $y_{i}$, $y_{j}$.}
\end{center}
\end{figure}

Since word-representable co-bipartite graphs follow the above conditions, we need to prove that it is sufficient to follow this condition to obtain a word-representable co-bipartite graph. To prove this, we first provide an orientation based on this ordering. In \( K_m \), the orientation of edges is defined as \( i \rightarrow j \) for \( i, j \in V(K_m) \) where \( i < j \). The orientation of the edges connecting vertices \( u \) and \( v \), with \( u \in V(K_m) \) and \( v \in V(K_n) \), is defined in Observation \ref{obs41} based on the vertex \( u \).
 Now we need to prove that if these conditions are satisfied, then this orientation is semi-transitive. Suppose $S$ is an orientation of the graph $\overline{B}(K_m,K_n)$ that satisfies the mentioned conditions. We assume that $S$ is not a semi-transitive orientation. So, there exists a shortcut in the orientation $S$. Suppose the induced subgraph of $v_1$, $v_2$, $\ldots$, $v_p$ vertices of the graph  $\overline{B}(K_m,K_n)$ creates a shortcut $P$.

In the shortcut $P$, vertices should be present from both $K_m$ and $K_n$. Otherwise, $P$ is a complete graph, not a shortcut. First, we check for each $v_i$, $1\leq i\leq p$ does there exist a $v_j$, $1\leq i<j\leq p$ such that $v_i\nsim v_j$. There should be at least one such $v_i$ that needs to be present in the shortcut $P$. We assume that $v_i$ and $v_j$ are the first two vertices we obtained in $P$ such that $v_i\nsim v_j$. Now, we consider the following two cases:

\textbf{Case 1:} Suppose $v_1= v_i$. Now, $v_1$ is adjacent to every $v_a$, $1\leq a<j$. So, we obtain a shortcut in the shortcut $P$ in the induced subgraph of the vertices $v_1$, $v_{j-1}$, $v_{j}$, $\ldots$ , $v_p$.
Now, we know that $v_1$ and $v_j$ are in different partitions in the graph $\overline{B}(K_m,K_n)$. Without loss of generality, we assume $v_1\in V(K_m)$ and $v_j\in V(K_n)$. So, the possible cases for $v_{j-1}$, $v_p$ are $\{v_{j-1},v_p\}\subseteq V(K_m)$ or $\{v_{j-1},v_p\}\subseteq V(K_n)$ or $v_{j-1}\in V(K_m), v_p\in V(K_n)$ or $v_{j-1}\in V(K_n), v_p\in V(K_m)$. In the following, we prove that this shortcut cannot exist for each case.

\textbf{Case 1.1:} If $\{v_{j-1},v_p\}\subseteq V(K_m)$, then there exists some $v_q\in V(K_n)$ and $v_r\in V(K_m)$, $j\leq q\leq n$, $j-1\leq r \leq p$, such that the orientation of the edge between $v_q$ and $v_r$ is $v_q\rightarrow v_r$, otherwise, $v_p\notin V(K_m)$. The following orientation occurs for the induced subgraph of the vertices $v_1$, $v_{j-1}$, $v_j$, $v_q$, $v_r$ and $v_p$.

\begin{figure}[H]
\begin{center}
\begin{tikzpicture}[node distance=1cm,auto,main node/.style={fill,circle,draw,inner sep=1pt,minimum size=5pt}]

\node [main node](1) {};
\node [left of=1,xshift=6mm] {$v_1$};
\node [main node,yshift=1mm](2) [ below of=1] {};
\node [left of=2,xshift=5mm] {$v_{j-1}$};
\node [main node,yshift=1mm](3) [ below of=2] {};
\node [left of=3,xshift=5mm] {$v_{r}$};
\node [main node,yshift=1mm](4) [ below of=3] {};
\node [left of=4,xshift=6mm] {$v_{p}$};
\node [main node](5) [ right of=2] {};
\node [right of=5,xshift=-5mm] {$v_j$};
\node [main node](6) [ right of=3] {};
\node [right of=6,xshift=-5mm] {$v_q$};

\draw[->] (1) -- (2);
\draw[->] (2) -- (3);
\draw[->] (3) -- (4);
\draw[->] (5) -- (6);
\draw[->] (2) -- (5);
\draw[->] (6) -- (3);
% \draw [->] (1) to [in=-720] (4);
% \draw [->] (1) to [in=-750, distance=1.5cm] (3);
\draw [->] (1) to [in=160, out=200, distance=1.2cm] (4);
\draw [->] (1) to [bend right=135, looseness=2.4] (3);   

\end{tikzpicture}
\caption{\label{fig8} Induced subgraph of the vertices $v_1$, $v_{j-1}$, $v_j$, $v_q$, $v_r$ and $v_p$.}
\end{center}
\end{figure}
The vertex $v_{j-1}$ is either type $A$ or type $C$, and the vertex $v_r$ is either type $B$ or type $C$. Therefore, $v_{j-1}\rightarrow v_j\rightarrow v_q\rightarrow v_r$ creates a shortcut. However, in Figures \ref{Case 4}(II) and \ref{Case 5}, it is shown that such a shortcut cannot exist if the vertices satisfy the mentioned conditions. Therefore, $v_j\rightarrow v_r$. 
Suppose $v_1$ is not adjacent to any of the vertices of $K_n$. Then, according to condition $6$, $v_{j-1}$ is a type $C$ vertex and there exists some type $A$ vertex $v'\in V(K_m)$, $v'\rightarrow v_1$. Let $N_{v'}\setminus V(K_m)=[x_1,y_1]$, $N_{v_{j-1}}\setminus V(K_m)=[1,x_{j-1}]\cup[y_{j-1},n]$. Suppose $v_r$ is a type $C$ vertex and $N_{v_r}\setminus V(K_m)=[1,x_r]\cup [y_r,n]$. But, according to the condition $4$, $1\leq v_j\leq x_r\leq y_1$ and as $y_{j-1}\leq v_j$, $y_{j-1}\leq y_1$. Therefore, according to the condition $6$, $v_1$ is not such a vertex which has no adjacent vertex in $V(K_n)$. Suppose $v_r$ is a type $B$ vertex and $N_{v_r}\setminus V(K_m)=[x_r,y_r]$. Then, according to condition $5$, $v_j\leq y_r\leq y_1$, and as $y_{j-1}\leq v_j$, $y_{j-1}\leq y_1$. Therefore, according to the condition $6$, $v_1$ is not such a vertex which has no adjacent vertex in $V(K_n)$.

Now, $v_1$ can be a type $A$ or type $C$ vertex. But, according to conditions $2$, $3$, $v_1$ should be neighbour with vertex $v_l\in V(K_n)$, $v_l<v_j$. Therefore, $v_{1}\rightarrow v_l\rightarrow v_j\rightarrow v_{j-1}$ creates a shortcut. But, in Figures \ref{Case 4}(II), \ref{Case 5}, it is shown that such a shortcut cannot exist if the vertices satisfy the mentioned conditions. Therefore, $v_1\rightarrow v_j$.

\textbf{Case 1.2:} If $\{v_{j-1},v_p\}\subseteq V(K_n)$, then the orientation occurs for the induced subgraph of the vertices $v_1$, $v_{j-1}$, $v_{j}$ and $v_p$ is shown in Figure \ref{fig9}.

\begin{figure}[H]
\begin{center}
\begin{tikzpicture}[node distance=1cm,auto,main node/.style={fill,circle,draw,inner sep=1pt,minimum size=5pt}]

\node [main node](1) {};
\node [right of=1,xshift=-5mm] {$v_{j-1}$};
\node [main node](2) [ below of=1] {};
\node [right of=2,xshift=-5mm] {$v_{j}$};
\node [main node](3) [ below of=2] {};
\node [right of=3,xshift=-5mm] {$v_{p}$};
\node [main node](4) [ left of=2] {};
\node [left of=4,xshift=5mm] {$v_1$};

\draw[->] (1) -- (2);
\draw[->] (2) -- (3);
\draw[->] (4) -- (1);
\draw[->] (4) -- (3);

\draw [->] (1) to [in=750, distance=1.5cm] (3);

\end{tikzpicture}
\caption{\label{fig9} Induced subgraph of the vertices $v_1$, $v_{j-1}$, $v_{j}$ and $v_p$.}
\end{center}
\end{figure}
We can see that $v_1$ is either type $A$ or $C$ vertex. But, according to the ordering of type $A$ or $C$, $N_{v_1}\setminus K_m=[x_{v_1},y_{v_1}]$ or $[1,x_{v_1}]\cup[y_{v_1},n]$. Therefore, $v_j\in [x_{v_1},y_{v_1}]$ or $[y_{v_1},n]$. So, $v_1$ and $v_j$ are adjacent. 

\textbf{Case 1.3:} If $v_{j-1}\in V(K_m), v_p\in V(K_n)$, then the following orientation occurs for the induced subgraph of the vertices $v_1$, $v_{j-1}$, $v_j$ and $v_p$.

\begin{figure}[H]
\begin{center}
\begin{tikzpicture}[node distance=1cm,auto,main node/.style={fill,circle,draw,inner sep=1pt,minimum size=5pt}]

\node [main node](1) {};
\node [right of=1,xshift=-5mm] {$v_{j}$};
\node [main node](2) [ below of=1] {};
\node [right of=2,xshift=-5mm] {$v_{p}$};
\node [main node](3) [ left of=1] {};
\node [left of=3,xshift=5mm] {$v_{1}$};
\node [main node](4) [ left of=2] {};
\node [left of=4,xshift=5mm] {$v_{j-1}$};

\draw[->] (1) -- (2);
\draw[->] (3) -- (4);
\draw[->] (3) -- (2);
\draw[->] (4) -- (1);

\end{tikzpicture}
\caption{\label{fig10} Induced subgraph of the vertices $v_1$, $v_{j-1}$, $v_{j}$ and $v_p$.}
\end{center}
\end{figure}
The vertices $v_1$ and $v_{j-1}$ can be type $A$ or $C$. But, according to conditions $2$, $3$ and $4$ this shortcut cannot occur as shown in Figures \ref{Case 2},\ref{Case 3a} and \ref{Case 4}(I). Therefore, $v_1$ and $v_j$ are adjacent. 
 
\textbf{Case 1.4:} If $v_{j-1}\in V(K_n), v_p\in V(K_m)$, then there exists some $v_q\in V(K_n)$ and $v_r\in V(K_m)$, $j\leq q\leq n$, $1\leq r \leq p$, such that the orientation of the edge between $v_q$ and $v_r$ is $v_q\rightarrow v_r$, otherwise, $v_p\notin V(K_m)$. The orientation occurs for the induced subgraph of the vertices $v_1$, $v_{j-1}$, $v_j$, $v_q$, $v_r$ and $v_p$ is shown in Figure \ref{fig11}.

\begin{figure}[H]
\begin{center}
\begin{tikzpicture}[node distance=1cm,auto,main node/.style={fill,circle,draw,inner sep=1pt,minimum size=5pt}]

\node [main node](1) {};
\node [right of=1,xshift=-5mm] {$v_{j-1}$};
\node [main node](2) [ below of=1] {};
\node [right of=2,xshift=-5mm] {$v_{j}$};
\node [main node](3) [ below of=2] {};
\node [right of=3,xshift=-5mm] {$v_{q}$};
\node [main node](4) [ left of=1] {};
\node [left of=4,xshift=5mm] {$v_1$};
\node [main node](5) [ left of=2] {};
\node [left of=5,xshift=5mm] {$v_r$};
\node[main node](6)[below of=5]{};
\node [left of=6,xshift=5mm] {$v_p$};
\draw[->] (1) -- (2);
\draw[->] (2) -- (3);
\draw[->] (4) -- (5);
\draw[->] (4) -- (1);
\draw[->] (3) -- (5);
\draw[->] (5) -- (6);

\draw [->] (1) to [in=750, distance=1.5cm] (3);
\draw [->] (4) to [in=750, distance=-1.5cm] (6);
\end{tikzpicture}
\caption{\label{fig11} Induced subgraph of the vertices $v_1$, $v_{j-1}$, $v_j$, $v_q$, $v_r$ and $v_p$.}
\end{center}
\end{figure}
In the figure, we can see there is a shortcut $v_1\rightarrow v_{j-1}\rightarrow v_q\rightarrow v_r$. The vertex $v_1$ can be either type $A$ or $C$, and the vertex $v_r$ can be either type $B$ or $C$. According to condition $2$, both $v_1$ and $v_r$ cannot be type $C$ because, $x_1\leq x_r$. According to conditions $4$(II) and $5$, this shortcut cannot occur as shown in Figures \ref{Case 4} and \ref{Case 5}. Therefore, $v_1\rightarrow v_q$.  However, the ordering of the neighbours of the vertex $v_1$ should be consecutive. Therefore, $v_1$ and $v_j$ are adjacent.

\textbf{Case 2:} Suppose $v_1\neq v_i$. Now, every $v_a$, $1\leq a<i$ is adjacent to every $v_b$, $1\leq b\leq p$ in the shortcut $P$. There exists at least one $v_a$ because, in Case 1, we proved that $v_1$ is adjacent to all the vertices present in the shortcut. Also, $v_i$ is adjacent to every $v_c$, $1\leq c<j$. So, we obtain the following shortcut in the shortcut $P$ among $v_{i-1}$, $v_i$, $v_{j-1}$ and $v_j$ vertices.
\begin{figure}[H]
\begin{center}
\begin{tikzpicture}[node distance=1cm,auto,main node/.style={circle,draw,inner sep=1pt,minimum size=5pt}]

\node (1) {$v_{i-1}$};
\node (2) [ right of=1] {$v_i$};
\node (3) [ right of=2] {$v_{j-1}$};
\node (4) [ right of=3] {$v_j$};

\draw[->] (1) -- (2);
\draw[->] (2) -- (3);
\draw[->] (3) -- (4);
\draw [->] (1) to [out=60, in=120] (4);
\draw [->] (1) to [out=60, in=120] (3);

\end{tikzpicture}

\caption{\label{fig12} Orientation among the vertices $v_{i-1}$, $v_i$, $v_{j-1}$ and $v_j$.}
\end{center}
\end{figure}

For this case, we can apply the same arguments as in Case 1 by treating the vertex $v_{i-1}$ as $v_1$ and the vertex $v_j$ as $v_p$ in Case 1. 
\end{proof}

\begin{theorem}
Let $G$ be a word-representable co-bipartite graph, and let $S$ be a semi-transitive orientation of $G$ containing vertices of types $A$, $B$, and $C$. Then $G$ admits another semi-transitive orientation in which every vertex is of either type $A$ or type $C$, and no vertex is of type $B$.
\end{theorem}

\begin{proof}
    We construct a new orientation by changing the type $B$ vertices into type $A$ vertices, and this orientation is also semi-transitive. Suppose $B=\{b_1,b_2,\ldots, b_l\}$ is the set of type $B$ vertices in the semi-transitive orientation $S$ such that $b_i<b_j$, $1\leq i<j\leq l$. We change the orientation of the edges as follows:
    \begin{itemize}
        \item If $u,v\in B$ or $u,v \notin B$, there is no change in the orientation $S$.
        \item If $u\in B$ and $v\notin B$, then the directed edge $v\rightarrow u$ is changed to  $u\rightarrow v$. 
    \end{itemize}

    From this construction, the type $B$ vertices become type $A$ vertices due to the change in the direction of edges connecting them to the opposite partition. We now show that this new orientation also satisfies the vertex ordering conditions mentioned in Theorem \ref{thmlab}. As there is no change in the type $A$ and type $C$ vertices, we only need to show that the new type $A$ vertices also satisfy the conditions of Theorem \ref{thmlab}. Suppose $a,b\in V(K_m)$ and $a\notin B$, $b\in B$. Since $b\rightarrow a$ is the new edge, it follows that $b<a$. It satisfies the first condition.
    
   Since there is no change in the orientation of edges between any two vertices $b_i, b_j \in B$ with $b_i < b_j$, let $N_{b_i} \setminus V(K_m) = [x_i, y_i]$ and $N_{b_j} \setminus V(K_m) = [x_j, y_j]$. Then $x_i \leq x_j$ and $y_i \leq y_j$. Now we need to show that the type $A$ vertices of the semi-transitive orientation $S$ also satisfy this condition with the new type $A$ vertices. Suppose $a_i$ is a type $A$ vertex in the orientation $S$, and $b_j$ is the type $B$ vertex in $S$, $N_{a_i}\setminus V(K_m)=[x_i,y_i]$, $N_{b_j}\setminus V(K_m)=[x_j,y_j]$. Then, $x_j\leq x_i$ and $y_j\leq y_i$. Now, in the new vertex ordering $b_j<a_i$, they satisfy the required condition for two type $A$ vertices, that is $x_j\leq x_i$ and $y_j\leq y_i$. Therefore, the new ordering satisfies the third condition.

   For the fourth condition, the vertices of types $A$ and $B$ also satisfy the same neighbourhood ordering condition. Therefore, changing the type $B$ vertices to type $A$ vertices also satisfies the fourth condition. Therefore, the new orientation is also semi-transitive.
\end{proof}

\section{Representation number}\label{sc4}
 In this section, we determine the maximum possible representation number of a co-bipartite graph. We show that this number is at most 3 by providing an algorithm that constructs a $3$-word representation for any word-representable co-bipartite graph. In this algorithm, we take the matrix $M_{m \times n}$ as input, described below.

$M_{m\times n}$ is a matrix representation of the co-bipartite graph $\overline{B}(K_m,K_n)$ where
\begin{itemize}
    \item the rows of $M_{m\times n}$ correspond to the vertices of $K_m$,
    \item the columns of $M_{m\times n}$ correspond to the vertices of $K_n$, and
    \item $M[i][j]=\begin{cases}
                    1, \text{ if } i\sim j\\
                    0, \text{ otherwise }
                \end{cases}$
\end{itemize}
In the matrix $M_{m\times n}$, we define two types of vertices, where $1\leq j\leq n$.
\begin{itemize}
    \item If $a$ is a type $A$ vertex then, it is represented as $[x_i,y_i]$ where 
    $M_{a\times j}=\begin{cases}
        1, \text{ if } x_i\leq j\leq y_i\\
        0, \text{ otherwise}
    \end{cases}$
    \item If $c$ is a type $C$ vertex then, it is represented as $[1,x_i]\cup [y_i,n]$ where
    $M_{c\times j}=\begin{cases}
        1, \text{ if } 1\leq j\leq x_i\\
        1, \text{ if } y_i\leq j\leq n \\
        0, \text{ otherwise}
    \end{cases}$
\end{itemize}
One key point to note is that any vertex $v$  with $ v \in V(K_m) $ and $ N_v \setminus V(K_m) = \varnothing $ is treated as a type $A$ vertex in the algorithm.

\begin{theorem}\label{repn}
    The representation number of word-representable co-bipartite graphs is at most $3$.
\end{theorem}   
\begin{proof}
    We provide an algorithm to generate a $3$-uniform word that represents the word-representable co-bipartite graph. The construction starts with three identical permutations of the vertices of $K_n$. Vertices of $K_m$ are inserted so that each vertex alternates exactly with its neighbourhood interval. Type $C$ vertices are inserted first to preserve the nesting of neighbourhoods, followed by type $A$ vertices inserted according to the ordering constraints of Theorem \ref{thmlab}. This guarantees a valid $3$-uniform word-representation.

\noindent\textbf{Algorithm: Word-representation of a word-representable co-bipartite graph}
\vspace{1mm}
\label{alg:my_algorithm}
\begin{algorithmic}[1]
 \State Input: $M_{m\times n}$ that satisfies the conditions for a word-representable co-bipartite graph.\\
\State Let $P$ be the permutation of the vertices of $K_n$ according to their transitive orientation.
\State $w=P_1P_2P_3$, where $P_1=P_2=P_3=P$.

\State Initialize empty dynamic arrays $A \gets [\;]$ and $C \gets [\;]$.

\For{each vertex $v \in V(K_m)$ in vertex ordering}
    \If{$N_v\setminus V(K_m)=[x,y]$}
        \State append $v$ to $A$
    \ElsIf{$N_v\setminus V(K_m)=[1,x]\cup[y,n]$}
        \State append $v$ to $C$
    \EndIf
\EndFor

\For{$i=1$ to length$(C)$}
    \State Let $c_i=C[i]$ and $N_{c_i}\setminus V(K_m)=[1,x_i]\cup[y_i,n]$
    \If{$i>1$ and $N_{c_{i-1}}\setminus V(K_m)=[1,x_i]\cup [y_{i-1},n]$} 
        \State replace $x_ic_{i-1}$ with $x_ic_{i-1}c_i$ in $P_1$ 
    \Else
        \State replace $x_i$ with $x_ic_i$ in $P_1$.
    \EndIf
    \State replace $x_i$ with $x_ic_i$ in $P_2$.
    \State replace $y_i$ with $c_iy_i$ in $P_3$.
\EndFor
\For{$i=1$ to length$(A)$}
    \State Let $a_i=A[i]$ and $N_{a_i}\setminus V(K_m)=[x_i,y_i]$
    \State Let $c_j=C[1]$ and $N_{c_j}\setminus V(K_m)=[1,x_j]\cup[y_j,n]$
    \If{$i>1$ and $N_{a_{i-1}}\setminus V(K_m)=[x_{i-1},y_i]$}
        \State replace $y_ia_{i-1}$ with $y_ia_{i-1}a_i$ in $P_1$.
    \Else
        \State replace $y_i$ with $y_ia_i$ in $P_1$.
    \EndIf
    \If{$x_i>x_j$}
        \State replace $x_jc_j$ with $x_j$ in $P_1$.
        \State replace $x_i$ with $a_ic_jx_i$ in $P_1$.
        \State replace $y_i$ with $y_ia_i$ in $P_2$.
    \Else
        \State Suppose $a_{i-1}=A[i-1]$ and $N_{a_{i-1}}\setminus V(K_m)=[x_{i-1},y_{i-1}]$.
        \If{$y_{i-1}< x_i$}
            \State replace $y_{i-1}a_{i-1}$ with $y_{i-1}$ in $P_1$.
            \State replace $x_i$ with $a_ia_{i-1}x_i$ in $P_1$.
            \State replace $y_i$ with $y_ia_i$ in $P_2$.
        \Else
            \State replace $x_i$ with $a_ix_i$ in $P_1$.
            \State replace $y_i$ with $y_ia_i$ in $P_2$.
        \EndIf
    \EndIf
\EndFor
\State Return $w$
\end{algorithmic}
Our claim is that the word obtained from the algorithm is a word-representation of the word-representable co-bipartite graph.\\
\textbf{Proof of Correctness:}\\
First, we check whether the edges of the co-bipartite graph are represented by the alternation in the word-representation or not.
\begin{itemize}
    \item The vertices of $K_n$ are adjacent to one another. The word $w_{V(K_n)}=P_1P_2P_3=P^3$. Therefore, the vertices of $K_n$ are alternating with one another in the word $w$.
    \item Suppose $a_i$ is a type $A$ vertex of $K_m$, and $N_{a_i}\setminus V(K_m)=[x_i,y_i]$. For a type $A$ vertex, according to the first if condition, $a_i$ occurs after the first $y_i$. Now, in the second condition, $a_i$ occurs before the first $x_i$ and after the second $y_i$. Therefore, the word $w_{\{a_i,x_i,\ldots,y_i\}}=a_ix_i\cdots y_ia_ix_i\cdots y_ia_ix_i\cdots y_i$. Therefore, $a_i$ is alternating with the vertices of $N_{a_i}\setminus V(K_m)$.
    \item Suppose $c_i$ is a type $C$ vertex of $K_m$, and $N_{c_i}\setminus V(K_m)=[1,x_i]\cup[y_i,n]$. For a type $C$ vertex, according to the first if condition, $c_i$ occurs after the first $x_i$. Then, $c_i$ occurs after the second $x_i$ and before the third $y_i$. Therefore, the word $w_{\{c_i,1,\ldots,x_i,y_i,\ldots ,n\}}=1\cdots x_ic_iy_i\cdots n1\cdots x_ic_iy_i\cdots n1\cdots x_ic_iy_i$ $\cdots n$.  Therefore, $c_i$ is alternating with the vertices of $N_{c_i}\setminus V(K_m)$.
    \item Suppose $a_1,\ldots,a_l$ are the type $A$ vertices and $c_1,\ldots,c_p$ are the type $C$ vertices of $K_m$. Therefore, each pair of such vertices should alternate with each other in the word $w$. We consider three cases to check the alternation in the following.
    \\\textbf{Case 1:} First, we check whether any two arbitrary type $A$ vertices are alternating with one another or not. Suppose $a_i$ and $a_j$ are two type $A$ vertices,$a_i<a_j$ and $N_{a_i}\setminus V(K_m)=[x_i,y_i]$, $N_{a_j}\setminus V(K_m)=[x_j,y_j]$. According to the third condition of Theorem \ref{thmlab}, $x_i\leq x_j$ and $y_i\leq y_j$. For this, two conditions occur that are described below.\\
    \textbf{Case 1.1:} If $y_i\geq x_j$, then $x_i\leq x_j\leq y_i\leq y_j$. Then the word $w_{\{x_i,x_j,y_i,y_j\}}=x_ix_jy_iy_j x_ix_jy_iy_j x_ix_jy_iy_j$. Now, the word $w_{\{a_i,a_jx_i,x_j,y_i,y_j\}}=a_ix_ia_jx_jy_ia_iy_j$ $a_j x_ix_jy_ia_iy_ja_j x_ix_jy_iy_j$. Therefore, $a_i$ and $a_j$ alternate with each other in this word.\\
    \textbf{Case 1.2:} If $y_i<x_j$, then $x_i\leq y_i<x_j\leq y_j$. Then the word $w_{\{x_i,x_j,y_i,y_j\}}=x_iy_ix_jy_j  x_iy_ix_jy_j  x_iy_ix_jy_j$. In this case, we remove the second occurrence of $a_i$ and place it after the first occurrence of $a_j$. Therefore, $w_{\{a_i,a_jx_i,x_j,y_i,y_j\}}=a_ix_iy_ia_ja_ix_jy_ja_j  x_iy_ia_ix_jy_ja_j  x_iy_ix_jy_j$. Therefore, $a_i$ and $a_j$ alternate with each other in this word.\\
    \textbf{Case 2:} We check whether any two arbitrary type $C$ vertices are alternating with one another or not. Suppose $c_i$ and $c_j$ are two type $C$ vertices,$c_i<c_j$ and $N_{c_i}\setminus V(K_m)=[1,x_i]\cup[y_i,n]$, $N_{c_j}\setminus V(K_m)=[1,x_j]\cup[y_j,n]$. According to the second condition of Theorem \ref{thmlab}, $x_i\leq x_j$ and $y_i\leq y_j$. We can see that $w_{\{x_i,y_i,x_j,y_j\}}= x_ix_jy_iy_j x_ix_jy_iy_j x_ix_jy_iy_j$ or $x_iy_ix_jy_j x_iy_ix_jy_j x_iy_ix_jy_j$. In both of the cases, the word $w_{\{c_i,c_j, x_i,y_i,x_j,y_j\}}=x_ic_ix_jc_jy_iy_j x_ic_ix_jc_jy_iy_j x_ix_j$ $c_iy_ic_jy_j$ or $x_ic_iy_ix_jc_jy_j  x_ic_iy_ix_jc_jy_j $ $x_ic_iy_ix_jc_jy_j$. Therefore, $c_i$ and $c_j$ alternate with each other in this word.

    \textbf{Case 3:} We check whether any two arbitrary type $A$ and $C$ vertices are alternating with one another or not. Suppose $a_i$ and $c_j$ are type $A$ and $C$ vertices respectively, and $N_{a_i}\setminus V(K_m)=[x_i,y_i]$, $N_{c_j}\setminus V(K_m)=[1,x_j]\cup[y_j,n]$. According to the fourth condition of Theorem \ref{thmlab}, $x_i\leq y_j$ and $x_j\leq y_i$. For this, two conditions occur that are described below.\\
    \textbf{Case 3.1:} If $x_i\leq x_j$, then $x_i\leq x_j\leq y_i\leq y_j$ or  $x_i\leq x_j\leq y_j\leq y_i$. Then the word $w_{\{x_i,x_j,y_i,y_j\}}=x_ix_jy_iy_j x_ix_jy_iy_j x_ix_jy_iy_j$ or $x_ix_jy_jy_i x_ix_jy_jy_i x_ix_jy_jy_i$. In both of the cases, the word $w_{\{a_i,c_j, x_i,y_i,x_j,y_j\}}=a_ix_ix_jc_jy_ia_iy_j x_ix_jc_jy_ia_iy_j$ $ x_ix_jy_ic_jy_j$ or $a_ix_ix_jc_jy_jy_ia_i x_ix_jc_jy_jy_ia_i x_ix_jc_jy_jy_i$. Therefore, $a_i$ and $c_j$ alternate with each other in this word.\\
    \textbf{Case 3.2:} If $x_i>x_j$, then $x_j<x_i\leq y_i\leq y_j$ or $x_j<x_i\leq y_j\leq y_i$. Then the word $w_{\{x_i,x_j,y_i,y_j\}}=x_jx_iy_iy_j x_jx_iy_iy_j x_jx_iy_iy_j$ or $x_jx_iy_jy_i x_jx_iy_jy_i x_jx_iy_jy_i$. In this case, we remove the first occurrence of $c_j$ and place it after the first occurrence of $a_i$. In both of the cases, the word 
    $w_{\{a_i,c_j, x_i,y_i,x_j,y_j\}}= x_ja_ic_jx_iy_ia_iy_j c_jx_jx_iy_ia_iy_j x_jx_iy_ic_jy_j$ or $x_ja_ic_jx_iy_jy_ia_i c_jx_jx_iy_jy_ia_i x_jx_ic_jy_jy_i$. Therefore, $a_i$ and $c_j$ alternate with each other in this word.
        
\end{itemize}

We now verify that the non-edges of the graph are represented by the non-alternation in the word-representation or not.
\begin{itemize}
    \item Suppose $a_i$ is a type $A$ vertex of $K_m$, and $N_{a_i}\setminus V(K_m)=[x_i,y_i]$. Then, the vertices $1,\ldots,x_{i-1},y_{i+1},$ $\ldots,n$ should not alternate with $a_i$. We already know that $w_{\{a_i,x_i,\ldots,y_i\}}=a_ix_i\cdots y_ia_ix_i\cdots y_ia_ix_i\cdots y_i$. There are no other letters between the first $a_i$ and the first $x_i$. Then, $w_{V(K_n)\cup\{a_i\}}=1\cdots x_{i-1}a_ix_i\cdots $ $y_ia_iy_{i+1}\ldots n 1\cdots x_{i-1}x_i\cdots y_ia_i y_{i+1}\ldots n 1\cdots x_{i-1}x_i\cdots y_iy_{i+1}$ $\cdots n$. Therefore, $1\cdots x_{i-1} a_i a_i $ $y_{i+1}\ldots n $ shows the non-alternation of the vertices $1,\ldots,x_{i-1},y_{i+1},\ldots,n$ with $a_i$. Therefore, $a_i$ is not adjacent to $1,\ldots,x_{i-1},y_{i+1},\ldots,n$.

    \item Suppose $c_i$ is a type $C$ vertex of $K_m$, and $N_{c_i}\setminus V(K_m)=[1,x_i]\cup[y_i,n]$. Then the vertices, $x_{i+1,\ldots,y_{i-1}}$ should not alternate with $c_i$. We already know that  $w_{\{c_i,1,\ldots,x_i,y_i,\ldots ,n\}}=1\cdots x_ic_iy_i\cdots n$ $1\cdots x_ic_iy_i\cdots n$ $1\cdots x_ic_iy_i\cdots n$. There exist two cases for the first occurrence of the $c_i$ in the word that are described below.
    
    \textbf{Case 1:} There is no other letters between the first $x_i$ and the first $c_i$. Then $w_{V(K_n)\cup \{c_i\}}=1\cdots x_ic_ix_{i+1}$ $\cdots y_{i-1} y_i\cdots n  1\cdots x_ic_ix_{i+1}\cdots y_{i-1}y_i\cdots n  1\cdots $ $x_ix_{i+1}\cdots y_{i-1}c_iy_i\cdots n$. Therefore, $c_ix_{i+1}\cdots$ $ y_{i-1}x_{i+1}\cdots y_{i-1}c_i$ shows the non-alternation of the vertices $x_{i+1}\cdots y_{i-1}$ with $c_i$. Therefore, $c_i$ is not adjacent to $x_{i+1}\cdots y_{i-1}$.
    
    \textbf{Case 2:} Suppose the first occurrence of the $c_i$ is after $x_j\in V(K_n)$, $x_i<x_j$. Then $w_{V(K_n)\cup \{c_i\}}=1\cdots x_ix_{i+1}\cdots x_jc_i\cdots y_{i-1} y_i\cdots n  1\cdots x_ic_ix_{i+1}\cdots y_{i-1}y_i$ $\cdots n  1\cdots x_ix_{i+1}\cdots y_{i-1}c_iy_i\cdots n$, because the second and third occurrence of $c_i$ is after second $x_i$ and before third $y_i$, respectively. In this case also, $c_ix_{i+1}\cdots y_{i-1}x_{i+1}\cdots y_{i-1}c_i$ shows the non-alternation of the vertices $x_{i+1}\cdots y_{i-1}$ with $c_i$. Therefore, $c_i$ is not adjacent to $x_{i+1}\cdots y_{i-1}$.
\end{itemize}
Therefore, the word obtained from the algorithm is a word-representation of the word-representable co-bipartite graph.
\end{proof}
It has been shown that permutation graphs are the only circle co-bipartite graphs. 
\begin{theorem}(\textit{\cite{srinivasan5128202bipartite}})\label{comp}
     If a bipartite graph G is the complement of a circle graph, then both G and its complement are permutation graphs.
\end{theorem}
According to Theorem \ref{2unf}, the representation number of circle graphs is $2$. Based on this, we obtain the following result.
\begin{cor}
   If the word-representable co-bipartite graph $G$ is not a permutation graph, then the representation number of the graph is $3$. 
\end{cor}
  
\begin{proof}
   According to Theorem \ref{comp}, a co-bipartite graph $G$ is a circle graph if and only if $G$ is a permutation graph. According to Theorem \ref{2unf}, only circle graphs have a representation number equal to $2$. Furthermore, according to Theorem \ref{repn}, we know that the representation number of any word-representable co-bipartite graph is at most $3$. Therefore, if $G$ is a word-representable co-bipartite graph but not a permutation graph, its representation number is $3$.
\end{proof}
After determining the representation number, we now study the number of labelled word-representable co-bipartite graphs on $n$ vertices. In the next subsection, we determine the speed and entropy of this class.
\section{Speed and entropy}\label{sc5}

In this section, we study the asymptotic enumeration of word-representable co-bipartite graphs using the notions of speed and entropy.

\begin{dnt}
Let $\mathcal{X}$ be a class of labelled graphs and let $\mathcal{X}_n$ denote the number of graphs in $\mathcal{X}$ on $n$ vertices. The function $\mathcal{X}_n$, viewed as a function of $n$, is called the \emph{speed} of the class $\mathcal{X}$.
\end{dnt}

\begin{dnt}
Let $\mathcal{X}$ be a class of labelled graphs with speed $\mathcal{X}_n$. The \emph{entropy} of the class $\mathcal{X}$ is defined as
\[
\lim_{n \to \infty} \frac{\log_2 \mathcal{X}_n}{\binom{n}{2}}= 1 - \frac{1}{k(X)}
\]
, where $k(X)$ is called the \emph{index} of the class $X$. The index is defined in terms of the graph classes ${\mathcal E}_{i,j}$, which consist of graphs whose vertex set can be partitioned into at most $i$ independent sets and $j$ cliques. The index $k(X)$ is the maximum integer $k$ such that $X$ contains the class ${\mathcal E}_{i,j}$ for some integers $i$ and $j$ satisfying $i + j = k$.
\end{dnt}

The speed and entropy measure the growth rate and structural complexity of a graph class. For example, the class of all graphs has speed $2^{\binom{n}{2}} = 2^{\Theta(n^2)}$ and entropy $1$. The class of co-bipartite graphs also has speed $2^{\Theta(n^2)}$ and entropy $\frac{1}{2}$, since co-bipartite graphs belong to the class $\mathcal{E}_{0,2}$ in the Alekseev–Bollobás–Thomason theorem. Similarly, it was shown in \cite{collins2017new} that the entropy of word-representable graphs is $\frac{2}{3}$. By the Alekseev--Bollobás--Thomason theorem \cite{alekseev1992range,bollobas1995projections}, it follows that the speed of word-representable graphs is $2^{\Theta(n^2)}$. 

We now show that the class of word-representable co-bipartite graphs grows much slower than both the class of all co-bipartite graphs and the class of all word-representable graphs.

\begin{theorem}
Let $\mathcal{W}$ be the class of word-representable co-bipartite graphs. Then the number $\mathcal{W}_n$ of labelled graphs in $\mathcal{W}$ on $n$ vertices satisfies
\[
\mathcal{W}_n = 2^{O(n\log n)}.
\]
In particular, the entropy of the class $\mathcal{W}$ is $0$.
\end{theorem}

\begin{proof}

Let $G \in \mathcal{W}$ be a word-representable co-bipartite graph on $n$ vertices. Then the vertex set of $G$ can be partitioned into two cliques $K_a$ and $K_b$ such that $a + b = n$ and $0 \le b \le a$.
Since all edges between vertices within each clique are present by definition, these edges are fixed and do not contribute to the counting. Therefore, it suffices to count the possible edge sets between the two cliques.

Fix an ordering of the vertices of $K_b$ as $V(K_b) = \{1,2,\dots,b\}$. By the vertex ordering characterisation, for each vertex $v \in K_a$, the neighbourhood of $v$ in $K_b$ is either of the form
\[
N(v)\setminus K_a = [x,y]
\quad \text{or} \quad
N(v)\setminus K_a = [1,x] \cup [y,b],
\]
for some integers $1 \le x \le y \le b$.

We first count the number of possible interval subsets. An interval $[x,y]$ is uniquely determined by choosing its left endpoint $x$ and right endpoint $y$ with $1 \le x \le y \le b$. For a fixed $x$, there are exactly $(b-x+1)$ choices for $y$. Therefore, the total number of interval subsets is
\[
\sum_{x=1}^{b} (b-x+1)
=
\frac{b(b+1)}{2}
=
O(b^2).
\]

Similarly, a circular interval subset of the form $[1,x] \cup [y,b]$ is uniquely determined by the pair $(x,y)$ satisfying $1 \le x \le y \le b$. Since there are at most $b^2$ such pairs, the number of circular interval subsets is also $O(b^2)$.

Thus, each vertex in $K_a$ has at most $c b^2 \le c n^2$ possible neighbourhoods, for some constant $c$.

Since the neighbourhood of each vertex in $K_a$ is determined by at most $O(b^2)$ possible intervals satisfying the vertex ordering conditions, the total number of possible neighbourhood assignments is at most $(c b^2)^a$. Therefore, for a fixed partition into cliques of sizes $a$ and $b$, the number of possible graphs is at most $(c b^2)^a$. For each choice of $a$, there are $\binom{n}{a}$ ways to select the vertices of the clique $K_a$ from the $n$ vertices. Thus, 
\[
\mathcal{W}_n
\le \sum_{a=0}^{n} \binom{n}{a} (c b^2)^a
\le \sum_{a=0}^{n} \binom{n}{a} (c n^2)^a
=
(1 + c n^2)^n.
\]

Taking logarithms,
\[
\log_2 \mathcal{W}_n
\le
n \log_2(1 + c n^2)
=
O(n \log n).
\]

Hence,
\[
\mathcal{W}_n
=
2^{O(n \log n)}.
\]

We now calculate the entropy of the class $\mathcal{W}$. First, we determine the index of the class $\mathcal{W}$. Since $\mathcal{W}$ contains all complete graphs, it contains the class ${\mathcal E}_{0,1}$, and hence $k(\mathcal{W}) \ge 1$. However, $\mathcal{W}$ does not contain the class ${\mathcal E}_{0,2}$, since not all co-bipartite graphs are word-representable. Therefore, $k(\mathcal{W}) = 1$.

Substituting $k(\mathcal{W}) = 1$ into the entropy formula
\[
\lim_{n \to \infty}
\frac{\log_2 \mathcal{W}_n}{\binom{n}{2}}
=
1 - \frac{1}{k(\mathcal{W})},
\]
we obtain
\[
\lim_{n \to \infty}
\frac{\log_2 \mathcal{W}_n}{\binom{n}{2}}
=
0.
\]
\end{proof}

\section{Conclusion}
In this paper, we study word-representable co-bipartite graphs using vertex orderings. We obtain necessary and sufficient conditions for a co-bipartite graph to be word-representable in terms of vertex orderings. Using this characterisation, we show that every word-representable co-bipartite graph has representation number at most $3$. We also determine the speed and entropy of the class of word-representable co-bipartite graphs. As future work, this ordering-based characterisation can be used to design a polynomial-time algorithm for recognising word-representable co-bipartite graphs. This vertex ordering characterisation can also be used to analyse graph parameters of this class.

\end{document}